\documentclass [10pt, a4paper] {article}
\usepackage[latin2]{inputenc}
\usepackage[T1]{fontenc}
\usepackage{amsfonts}
\usepackage{mathrsfs}
\usepackage{amsmath}
\usepackage{amssymb}
\usepackage{graphicx}
\usepackage{amsthm}
\usepackage{color}
\usepackage[left=4 cm, right=4 cm]{geometry}
\usepackage[colorlinks=true]{hyperref}

\newcommand{\N}{\mathbb{N}}
\newcommand{\Z}{\mathbb{Z}}
\newcommand{\R}{\mathbb{R}}
\newcommand{\D}{\mathbb{D}}

\newcommand{\Lat}{\mathrm{Lat}\,}
\newcommand{\HLat}{\mathrm{HLat}\,}

\newcommand{\irB}{\mathcal{B}}
\newcommand{\irC}{\mathcal{C}}

\newcommand{\irE}{\mathcal{E}}
\newcommand{\irF}{\mathcal{F}}
\newcommand{\irH}{\mathcal{H}}
\newcommand{\irK}{\mathcal{K}}
\newcommand{\irL}{\mathcal{L}}

\newcommand{\irN}{\mathcal{N}}

\newcommand{\irR}{\mathcal{R}}

\newcommand{\irX}{\mathcal{X}}

\newtheorem{thm}{Theorem}
\newtheorem{prop}{Proposition}
\newtheorem{rem}{Remark}
\newtheorem{lem}{Lemma}
\newtheorem{cor}{Corollary}
\newtheorem{exmpl}{Example}

\begin{document}

\begin{center}
\textbf{\LARGE Positive operators arising asymptotically \\ from contractions}

\bigskip

{\Large György Pál Gehér}%\footnote{Supported by the TÁMOP-4.2.2/B-10/1-2010-0012 project.}}

\bigskip

%Dedicated to the memory of Béla Sz.-Nagy on the occasion of his 100th anniversary
\end{center}

\bigskip

\begin{abstract}
In this paper we characterize those positive operators which are asymptotic limits of contractions in strong operator topology or uniform topology. We examine the problem when the asymptotic limits of two contractions coincide.
\end{abstract}

\let\thefootnote\relax

\section{Introduction}

Let $\irH$ be a complex Hilbert space and let $\irB(\irH)$ stand for the C*-algebra of bounded, linear operators on $\irH$. For an operator $T\in\irB(\irH)$ we denote the null space, the range, the spectrum, the point spectrum and the essential spectrum of $T$ by $\irN(T)$, $\irR(T)$, $\sigma(T)$, $\sigma_p(T)$ and $\sigma_e(T)$, respectively. The spectral radius and the essential spectral radius are $r(T)$ and $r_e(T)$, respectively. We write $I$ for the identity operator, and if $\irX$ is a subspace of $\irH$ (i.e. closed linear manifold), then $I_\irX$ stands for the identity operator on $\irX$. The set of natural numbers and the set of non-negative integers are denoted by $\N$ and $\Z_+$, respectively.

If the subspace $\irX$ of $\irH$ has the property $T\irX\subset\irX$, we say that it is invariant for $T$, and write $\irX\in\Lat(T)$. If $\irX$ is invariant for every operator $C\in\irB(\irH)$ which commutes with $T$: $CT=TC$, then we call it a hyperinvariant subspace of $T$, in notation: $\irX\in\HLat(T)$. If $\irX$ is invariant for both $T$ and $T^*$, where $T^*$ denotes the adjoint of $T$, then $\irX$ is reducing for $T$. This means exactly that $\irX$ and its orthogonal complement: $\irX^\perp=\irH\ominus\irX$ are both invariant for $T$. If an operator has only trivial reducing subspaces (i.e. $\irH$ and $\{0\}$), then we call it an irreducible operator.

The operator $T\in\irB(\irH)$ is a contraction if $\|T\|\leq 1$. The vector $h\in\irH$ is \textit{stable} for the contraction $T\in\irB(\irH)$, if the orbit of $h$ converges to 0, i.e. $\lim_{n\to\infty} \|T^n h\| = 0$. The set of all stable vectors will be denoted by $\irH_0=\irH_0(T)$. We recall that $\irH_0$ is a hyperinvariant subspace of $T$, what we call the \textit{stable subspace of $T$}. The contractions can be classified according to the asymptotic behaviour of their iterates and the iterates of their adjoints. Namely, $T$ is \textit{stable} or \textit{of class $C_{0\cdot}$}, when all vectors are stable, in notation: $T\in C_{0\cdot}(\irH)$. If the stable subspace consists only of the null vector, then $T$ is \textit{of class $C_{1\cdot}$} or $T\in C_{1\cdot}(\irH)$. In the case when $T^*\in C_{i\cdot}(\irH)$ ($i=0$ or 1), we say that $T$ is \textit{of class $C_{\cdot i}$} or $T\in C_{\cdot i}(\irH)$. Finally, the class $C_{ij}(\irH)$ stands for the intersection $C_{i\cdot}(\irH)\cap C_{\cdot j}(\irH)$.

From now on, $T\in\irB(\irH)$ always denotes a contraction. Let us consider the sequence $\{T^{*n}T^n\}_{n=1}^\infty$ of positive operators, which is decreasing, so it has a limit in the strong operator topology (SOT):
\[ A_T := \lim_{n\to\infty}T^{*n}T^n. \]
We say that \textit{$A_T$ arises asymptotically from $T$}, or $A_T$ is the \textit{asymptotic limit} of $T$. In the case when this convergence holds in norm, we say that $A_T$ \textit{arises asymptotically from $T$ in uniform convergence} or $A_T$ is the \textit{uniform asymptotic limit} of $T$. We recall that $A_T^{1/2}$ acts as an intertwining mapping in a canonical realization of the unitary and isometric asymptote of the contraction $T$. Here we give the categorical definition of the latter concept. The pair $(X,V)$ is a \textit{contractive (unitary/isometric) intertwining pair} for $T$ if $V\in\irB(\irK)$ is unitary/isometric, $X\in\irB(\irH,\irK), \|X\| \leq 1$ and $X$ intertwines $T$ with $V$: $XT=VX$. A contractive (unitary/isometric) intertwining pair is called \textit{unitary/isometric asymptote} of $T$ if for any other contractive (unitary/isometric) intertwining pair $(Y,U)$ there exists a unique operator $Z$ such that $ZV=UZ$, $\|Z\|\leq 1$ and $Y=ZX$. 

For any contraction $T\in\irB(\irH)$ there exists a unique isometry $V_T\in\irB(\irR(A_T)^-)$ such that $A_T^{1/2}T=V_TA_T^{1/2}$. The pair $(X^+_T,V_T)$ is a realization of the isometric asymptote of $T$, where $X^+_T\in\irB(\irH,\irR(A_T)^-), X^+_Th = A_T^{1/2}h$. If the minimal unitary extension of $V_T$ is $W_T\in\irB(\irK)$, and $X_T\in\irB(\irH,\irK), X_Th = A_T^{1/2}h$, then $(X_T,W_T)$ is a unitary asymptote of $T$. For a detailed study of unitary asymptotes, including other useful realizations (e. g. with the *-residual part of the minimal unitary dilation), we refer to Chapter IX in \cite{NFBK}, \cite{KerchyDouglasAlg} and \cite{KerchyUnitaryAs}.

The main goal of this paper is to give a complete characterization of those positive operators that are asymptotic limits of contractions. We give necessary conditions in the next section, and prove the complete characterization in Section 3. In Section 4 we investigate some connections between contractions which have the same asymptotic limit. We shall use the isometric asymptote and the Sz.-Nagy--Foias functional calculus for this purpose.

\section{Necessary conditions}

First we give some necessary conditions. The following proposition is well known and provides some basic information on $A_T$.

\begin{prop}
For every contraction $T\in\irB(\irH)$ the following statements hold:
\begin{itemize}
\item[\textup{(i)}] $A_T$ is a positive contraction,
\item[\textup{(ii)}] $\irN(A_T)=\irH_0(T):=\{x\in\irH\colon \lim_{n\to\infty}\|T^nx\|=0\}$ is the hyperinvariant subspace of stable vectors,
\item[\textup{(iii)}] $\irN(A_T-I)=\irH_1(T):=\{x\in\irH\colon \lim_{n\to\infty}\|T^nx\|=\|x\|\}$ is the largest invariant subspace where $T$ is an isometry,
\item[\textup{(iv)}] $A_T=0$ if and only if $T\in C_{0\cdot}(\irH)$,
\item[\textup{(v)}] $0\notin\sigma_p(A_T)$ if and only if $T\in C_{1\cdot}(\irH)$.
\end{itemize}
\end{prop}

For the proof see \cite{NFBK} or \cite{Kubrusly}. Further necessary conditions are formulated in the next theorem.

\begin{thm} \label{nec_thm_contr}
If the positive operator $0\leq A\leq I$ is the asymptotic limit of a contraction $T$, then one of the following three possibilities occurs:
\begin{itemize}
\item[\textup{(i)}] $A=0$;
\item[\textup{(ii)}] $A$ is a non-zero finite rank projection, $\irH=\irH_0(T)\oplus\irH_1(T)$ and $\dim{\irH_1(T)}\in\N$;
\item[\textup{(iii)}] $\|A\| = r_e(A) = 1$.
\end{itemize}
\end{thm}

\begin{proof}
Suppose that $A\neq 0$. We will show that $\|A\|=1$ in this case. To this end set $v\in\irH\setminus\irH_0(T)$, $\|v\|=1$; then
\[ \|A^{1/2}v\|^2 = \big<A v, v\big> = \lim_{n\to\infty}\|T^n v\|^2 = \delta \in (0,1]. \]
For all $k\in\N$, we have
\[ \|A^{1/2}T^k v\| = \lim_{n\to\infty}\|T^{n+k} v\| = \delta^{1/2}, \]
and so
\[ \lim_{k\to\infty} \frac{\|A^{1/2}T^k v\|}{\|T^k v\|} = 1, \]
which means that $\|A^{1/2}\| = 1$, thus $\|A\| = 1$. For an alternative proof of this fact see Proposition 3.1 in \cite{Kubrusly}.

Now assume that (i) and (iii) don't hold. Then 1 is an isolated point of $\sigma(A)$ and $\dim{\irN(A-I)}=\dim{\irH_1(T)}\in\N$. Since $T$ is an isometry on the finite dimensional invariant subspace $\irH_1(T)$, it must be unitary on it. Therefore $\irH_1(T)$ is reducing for $T$. Now consider the decomposition $T=T'\oplus U$ where $U=T|\irH_1(T)$. Evidently $A=A_{T'}\oplus I_{\irH_1(T)}$ and $\sigma(A_{T'})\subset [0,1)$ which means $\|A_{T'}\|<1$. This yields $A_{T'}=0$ and $\irH_0(T)=\irH_1(T)^{\perp}$, so $A$ is a projection of finite rank $\dim\irH_1$, thus (ii) holds.
\end{proof}

We collect some additional properties of the asymptotic limit in the next remark.

\begin{rem}\label{contr_rem}
\textup{(a) For any set $\{T_\xi\in\irB(\irH_\xi)\}_{\xi\in \Xi}$ of contractions the equation $A_{T}=\sum_{\xi\in \Xi}\oplus A_{T_\xi}$ is obvious, provided $T=\sum_{\xi\in \Xi}\oplus T_\xi$.}

\textup{(b) It is easy to see that in a triangular decomposition
\[ T = \left[\begin{matrix}
T_1 & T_{12}\\
0 & T_2
\end{matrix}\right] \in \irB(\irH'\oplus\irH'') \]
$A_T = A_{T_1}\oplus A_{T_2}$ does not hold in general. A counterexample can be given by the contractive bilateral weighted shift defined by
\[ Te_k = \left\{\begin{matrix}
e_{k+1} & \text{ for } k>0\\
\frac{1}{2}e_{k+1} & \text{ for } k\leq 0\\
\end{matrix}\right. , \]
where $\{e_k\}_{k\in\Z}$ is an orthonormal basis in $\irH$. Indeed, an easy calculation shows
\[ A_Te_k = \left\{\begin{matrix}
e_k & \text{ for } k>0\\
(\frac{1}{2})^{-2k+2}e_k & \text{ for } k\leq 0\\
\end{matrix}\right. . \]
On the other hand, $\irH_1(T)=\vee_{k>0}\{e_k\}$, and the matrix of $T$ is
\[ T = \left[\begin{matrix}
T_1 & T_{12}\\
0 & T_2
\end{matrix}\right]\]
in the  decomposition $\irH_1(T)\oplus\irH_1(T)^\perp$. Here $T_2\in C_{0\cdot}(\irH_1(T)^\perp)$, so $A_{T_1}\oplus A_{T_2} = I\oplus 0$ is a projection, but $A_T$ is not.}

\textup{(c) The asymptotic limit $A_T$ is a projection if and only if $\irH=\irH_0(T)\oplus\irH_1(T)$, see Section 5.3 in \cite{Kubrusly}. Therefore if $A_T$ is a projection, then $\irH_1(T)$ is reducing for $T$.}

\textup{(d) If $T$ is an irreducible contraction on $\irH$ with $\dim\irH\geq 2$ and the asymptotic limit is a projection, then either $T$ is stable or $T$ is a simple unilateral shift. Indeed, in that case $T$ is either stable or an isometry by (c), but the simple unilateral shift is the only irreducible isometry.}

\textup{(e) If $T$ is a contraction and $U$ is an isometry, then $UTU^*$ is also a contraction and
\[A_{UTU^*}= \lim_{n\to\infty} (UTU^*)^{*n}(UTU^*)^n = UA_TU^*.\]
In particular, two unitarily equivalent contractions have unitarily equivalent asymptotic limits.}

\textup{(f) It can be easily seen that $A_T$ is a projection when $\irH$ is finite dimensional.}
\end{rem}

\section{Sufficiency}

In this section we prove our main theorem. Let $A\in\irB(\irH)$ be an arbitrary positive operator satisfying $0\leq A \leq I$. If $A$ is a projection, then $A$ arises asymptotically from itself, for example. But what can we say about the other cases? In order to give complete characterization, we need two lemmas. We shall write $\underline{r}(A)$ for the minimal element of $\sigma(A)\subset[0,1]$.

\begin{lem}\label{block_diag_same_dim_lemma}
Suppose that the block-diagonal positive contraction $A = \sum_{j=0}^\infty\oplus A_j\in\irB(\irH = \sum_{j=0}^\infty\oplus\irX_j)$ has the following properties:
\begin{itemize}
\item[\textup{(i)}] $\dim{\irX_j}=\dim{\irX_0}>0$ for every $j$,
\item[\textup{(ii)}] $r(A_j) \leq \underline{r}(A_{j+1})$ for every $j$,
\item[\textup{(iii)}] $A_1$ is invertible, and
\item[\textup{(iv)}] $r(A_j)\nearrow 1$.
\end{itemize} 
Then $A$ arises asymptotically from a $C_{\cdot 0}$-contraction in uniform convergence. 
\end{lem}

\begin{proof} 
Let us consider the unilateral shift $S\in\irB(\irH)$, given by
\[ S (x_0\oplus x_1 \oplus x_2 \oplus\dots) = 0\oplus U_0 x_0\oplus U_1 x_1 \oplus \dots \quad (x_j\in\irX_j), \]
where $U_j\colon \irX_j\to\irX_{j+1}$ are unitary transformations ($j\in\Z_+$). Let $T$ be defined by
\[ T|\irX_j = A_{j+1}^{-1/2}SA_{j}^{1/2} \quad (j\in\Z_+).\]
Since 
\[ \|A_{j+1}^{-1/2}SA_{j}^{1/2}x_{j}\| \leq \sqrt{\frac{1}{\underline{r}(A_{j+1})}} \|A_{j}^{1/2}x_{j}\| \leq \sqrt{\frac{r(A_{j})}{\underline{r}(A_{j+1})}} \|x_{j}\| \leq \|x_{j}\|, \]
we obtain that $T$ is a contraction of class $C_{\cdot 0}$. An easy calculation shows that
\[ T^{*n}T^n = \sum_{j=0}^\infty\oplus A_j^{1/2}S^{*n}A_{j+n}^{-1}S^nA_j^{1/2}. \]
By the spectral mapping theorem, we have
\[ \|A_j-A_j^{1/2}S^{*n}A_{j+n}^{-1}S^nA_j^{1/2}\| \leq \|A_j^{1/2}\|\cdot\|I_{\irX_j}-S^{*n}A_{j+n}^{-1}S^n\|\cdot\|A_j^{1/2}\| \] \[ \leq r(A_{j+n}^{-1}-I_{\irX_{j+n}}) \leq \frac{1}{\underline{r}(A_{j+n})}-1 \leq \frac{1}{\underline{r}(A_{n})}-1. \]
This yields
\[ \|T^{*n}T^n-A\| = \sup\left\{\|A_j-A_j^{1/2}S^{*n}A_{j+n}^{-1}S^nA_j^{1/2}\|\colon j\in\Z_+\right\} \] 
\[ \leq \frac{1}{\underline{r}(A_{n})}-1 \longrightarrow 0 \quad (n\to\infty). \]
So $A$ arises asymptotically from $T\in C_{\cdot 0}(\irH)$ in uniform convergence.
\end{proof}

The following lemma deals with diagonal positive contractions.

\begin{lem}\label{diag_clust_to_1_lemma}
Let $A$ be a positive diagonal contraction on a separable infinite dimensional Hilbert space $\irH$. Suppose that the eigenvalues of $A$ can be arranged into an increasing sequence $\{\lambda_j\}_{j=1}^\infty$, each listed according to its multiplicity, so that $0<\lambda_j<1$ holds for every $j\in\N$ and $\lambda_j\nearrow 1$. Then $A$ is the uniform asymptotic limit of a $C_{\cdot 0}$-contraction.
\end{lem}

\begin{proof} First we form a matrix $[\alpha_{l,m}]_{l,m\in\N}$ from the eigenvalues in the following way: $\alpha_{1,1}=\lambda_1$; $\alpha_{2,1}=\lambda_2$ and $\alpha_{1,2}=\lambda_3$; $\alpha_{3,1}=\lambda_4$, $\alpha_{2,2}=\lambda_5$ and $\alpha_{1,3}=\lambda_6$; \dots and so on. Hence we have:

\[ \left[
\begin{matrix}
\alpha_{1,1} & \alpha_{1,2} & \alpha_{1,3} & \alpha_{1,4} & \dots\\
\alpha_{2,1} & \alpha_{2,2} & \alpha_{2,3} \\
\alpha_{3,1} & \alpha_{3,2} \\
\alpha_{4,1} & & &  \\
\vdots & & & & \ddots \\
\end{matrix}\right] = 
\left[\begin{matrix}
\lambda_1 & \lambda_3 & \lambda_6 & \lambda_{10} & \dots\\
\lambda_2 & \lambda_5 & \lambda_9 \\
\lambda_4 & \lambda_8 \\
\lambda_7 & & &  \\
\vdots & & & & \ddots \\
\end{matrix}\right] \]
We can choose an orthonormal basis $\{e_{l,m}\colon l,m\in\N\}$ in $\irH$ such that $e_{l,m}$ is an eigenvector corresponding to the eigenvalue $\alpha_{l,m}$ of $A$. Now we form the subspaces:
\[ \irX_m := \vee\{e_{l,m}\colon l\in\N\} \quad (m\in\N), \]
which are reducing for $A$. For any $m\in\N$, we set $A_m:=A|\irX_m$. Let us consider also the unilateral shift $S\in\irB(\irH)$, defined by $Se_{l,m}=e_{l,m+1}$ ($l,m\in\N$). Now the operator $T\in\irB(\irH)$ is given by the following equality:
\[ T|\irX_m = A_{m+1}^{-1/2}SA_m^{1/2} \quad (m\in\N). \]
Since 
\[ T e_{l,m} = \sqrt{\frac{\alpha_{l,m}}{\alpha_{l,m+1}}}e_{l,m+1} \quad \text{and} \quad \sqrt{\frac{\alpha_{l,m}}{\alpha_{l,m+1}}}\leq 1 \quad (l,m\in\N), \]
$T$ is a $C_{\cdot 0}$-contraction. Furthermore, for every $l,m,n\in\N$, we have $\lambda_n\leq \alpha_{l,m+n}$, and so
\[ \|T^{*n}T^{n} e_{l,m} - A e_{l,m} \| = \frac{\alpha_{l,m}}{\alpha_{l,m+n}}-\alpha_{l,m} \leq \frac{1}{\lambda_{n}}-1 \to 0 \quad (n\to\infty). \]
Since $e_{l,m}$ is an eigenvector for both $A$ and $T^{*n}T^{n}$, the sequence $T^{*n}T^{n}$ uniformly converges to $A$ on $\irH$. So $A$ arises asymptotically from a $C_{\cdot 0}$-contraction in uniform convergence.
\end{proof}

Now we are ready to prove our main theorem. This states that a positive contraction, which acts on a separable space, is an asymptotic limit of a contraction if and only if one of the conditions (i)--(iii) of Theorem \ref{nec_thm_contr} holds. The non-separable case is a little bit more complicated, and will be handled after the main theorem. In what follows, $E$ stands for the spectral measure of the positive operator $A$ and $\irH(\omega)=E(\omega)\irH$ for any Borel subset $\omega\subset\R$. Let us consider the orthogonal decomposition $\irH=\irH_{d}\oplus\irH_c$, reducing for $A$, where $A|\irH_{d}$ is diagonal and $A|\irH_{c}$ has no eigenvalue. Let us denote the spectral measure of $A|\irH_{d}$ and $A|\irH_{c}$ by $E_d$ and $E_c$, respectively. For any Borel set $\omega\subset\R$ we shall write $\irH_c(\omega)=E_c(\omega)\irH_c$ and $\irH_{d}(\omega)=E_{d}(\omega)\irH_d$.

\begin{thm}\label{mainthm_separable}
Let $\irH$ be a separable, infinite dimensional Hilbert space, and $A$ a positive contraction acting on $\irH$. The following four conditions are equivalent:
\begin{itemize}
\item[\textup{(i)}] $A$ arises asymptotically from a contraction,
\item[\textup{(ii)}] $A$ arises asymptotically from a contraction in uniform convergence,
\item[\textup{(iii)}] $r_e(A)=1$ or $A$ is a projection of finite rank, and
\item[\textup{(iv)}] $\dim\irH((0,1]) = \dim\irH((\delta,1])$ for every $0\leq\delta<1$.
\end{itemize}
Moreover if one of these conditions holds and $\dim\irN(A-I)=0$ or $\aleph_0$, then the inducing $T$ can be chosen to be a $C_{\cdot 0}$-contraction.
\end{thm}

\begin{proof} The implication (i)$\Longrightarrow$(iii) follows from Theorem \ref{nec_thm_contr}, and (ii)$\Longrightarrow$(i) is trivial. First we prove the implication (iii)$\Longrightarrow$ (ii), in order to complete the implication circle (i)$\Longrightarrow$(iii)$\Longrightarrow$(ii)$\Longrightarrow$(i), and in the end of the proof we show the equivalence (iii)$\iff$(iv). We suppose that $r_e(A)=1$. (If $A$ is a finite rank projection, then $T=A$ can be chosen.) If $\irN(A)\neq\{0\}$, then $A$ has the form $A=0\oplus A_1$ in the decomposition $\irH=\irN(A)\oplus\irN(A)^\perp$, where $r_e(A_1)=1$. If $A_1$ arises asymptotically from the contraction $T_1$ in uniform convergence, then $A$ arises asymptotically from $0\oplus T_1$ in uniform convergence. Hence we may assume that $\irN(A)=\{0\}$. Obviously, one of the next three cases occurs.

\bigskip

\textit{Case 1: there exists a strictly increasing sequence $0=a_0<a_1<a_2<\dots$ such that $a_n\nearrow 1$ and $\dim\irH([a_n,a_{n+1}))=\aleph_0$ for every $n\in\Z_+$.}
If 1 is not an eigenvalue of $A$, then Lemma \ref{block_diag_same_dim_lemma} can be applied. So we may suppose that $\dim\irN(A-I)\geq 1$. In this case we have the orthogonal decomposition: $A = A_0\oplus A_1$, where $A_0 = A|\irN(A-I)^\perp$ and $A_1 = A|\irN(A-I)$. Again using Lemma \ref{block_diag_same_dim_lemma} we obtain a contraction $T_0\in\irB(\irN(A-I)^\perp)$ such that the uniform asymptotic limit of $T_0$ is $A_0$. Choosing any isometry $T_1\in\irB(\irN(A-I))$, $A$ arises asymptotically from $T:=T_0\oplus T_1$ in uniform convergence.

\bigskip

\textit{Case 2: $\irN(A-I) = \{0\}$ and there is no strictly increasing sequence $0=a_0<a_1<a_2<\dots$ such that $a_n\nearrow 1$ and $\dim\irH([a_n,a_{n+1}))=\aleph_0$ for every $n\in\Z_+$.} If $\dim\irH([0,\beta))<\aleph_0$ for each $0<\beta<1$, then $A$ is diagonal, all eigenvalues are in (0,1) and have finite multiplicities. Therefore Lemma \ref{diag_clust_to_1_lemma} can be applied. If this is not the case, then there is a $0<b<1$ which satisfies the following conditions: $\dim\irH([0,b))=\aleph_0$ and $\dim\irH([b,\beta))<\aleph_0$ for all $b<\beta<1$. We take the decomposition $\irH = \irH([0,b))\oplus\irH([b,1))$, where $\dim\irH([b,1))=\aleph_0$ obviously holds, since $1\in\sigma_e(A)$. In order to handle this case, we have to modify the argument applied in Lemma \ref{diag_clust_to_1_lemma}.

Let us arrange the eigenvalues of $A$ in $[b,1)$ in an increasing sequence $\{\lambda_j\}_{j=1}^\infty$, each listed according to its multiplicity. We form the same matrix $[\alpha_{l,m}]_{l,m\in\N}$ as in Lemma \ref{diag_clust_to_1_lemma}, and take an orthonormal basis $\{e_{l,m}\colon l,m\in\N\}$ in $\irH([b,1))$ such that each $e_{l,m}$ is an eigenvector corresponding to the eigenvalue $\alpha_{l,m}$ of $A$. Let $\irX_0:=\irH([0,b))$ and $\irX_m:=\vee\{e_{l,m}\colon l\in\N\}$ ($m\in\N$). Take an arbitrary orthonormal basis $\{e_{l,0}\}_{l=1}^\infty$ in the subspace $\irX_0$. We define the operator $T$ by the following equation
\[ T|\irX_m = A_{m+1}^{-1/2}SA_{m}^{1/2}|\irX_m \quad (m\in\Z_+),\]
where $A_m:=A|\irX_m$ and $S\in\irB(\irH)$, $Se_{l,m} = e_{l,m+1}$ ($l\in\N, m\in\Z_+$).

For a vector $x_0\in\irX_0$ we have
\[ \|Tx_0\| = \|A_1^{-1/2}SA_0^{1/2}x_0\| \leq \sqrt{\frac{1}{b}}\|A_0^{1/2}x_0\|\leq \|x_0\|, \]
so $T$ is a contraction on $\irX_0$. But it is also a contraction on $\irX_0^\perp$ (see the proof of Lemma \ref{diag_clust_to_1_lemma}), and since 
$T\irX_0\perp T(\irX_0^\perp)$, it is a contraction on the whole $\irH$.

We have to show yet that $T^{*n}T^n$ converges uniformly to $A$ on $\irX_0$. For $x_0\in\irX_0, \|x_0\|=1$ we get
\[ \|T^{*n}T^{n} x_0 - A x_0\| = \|A_0^{1/2}S^{*n}(A_n^{-1}-I_{\irX_n})S^{n}A_0^{1/2} x_0\| \]
\[ \leq \|A_n^{-1} - I_{\irX_n}\| < \frac{1}{\lambda_{n}}-1\to 0. \]
So $A$ arises asymptotically from $T$ in uniform convergence.

\bigskip

\textit{Case 3: $\dim{\irN(A-I)}>0$.} If $\dim{\irN(A-I)}<\aleph_0$, then we take the orthogonal decomposition $\irH = \irN(A-I)^\perp\oplus\irN(A-I)$. Trivially $1\in\sigma_e(A|\irN(A-I)^\perp)$. By Cases 1 and 2, we can find a contraction $T_0\in\irB(\irN(A-I)^\perp)$ such that the uniform asymptotic limit of $T=T_0\oplus I_{\irN(A-I)}$ is $A$.

If $\dim{\irN(A-I)}=\aleph_0$ and $A\neq I$, then we take an orthogonal decomposition $\irN(A-I) = \sum_{j=1}^\infty\oplus \irX_j$, where $\dim\irX_j=\dim\irN(A-I)^\perp$, and apply Lemma \ref{block_diag_same_dim_lemma}. If $A=I$ then just take an isometry for $T$.

\bigskip

Now we turn to the equivalence (iii)$\iff$(iv). If $A$ is a projection, then $\dim((\delta,1])$ is the rank of $A$ for every $0\leq\delta< 1$. If $1\in\sigma_e(A)$, then $\dim\irH((\delta,1])=\aleph_0$ holds for all $0\leq\delta< 1$. Conversely if $\dim\irH((\delta,1])=\dim\irH((0,1])$ is finite ($0\leq\delta< 1$), then obviously $A$ is a projection of finite rank. If this dimension is $\aleph_0$, then clearly $1\in\sigma_e(A)$.

Finally, from the previous discussions we can see that if the equivalent conditions (i)-(iv) hold, then the contraction $T$, inducing $A$, can be chosen from the class $C_{\cdot 0}$ provided $\dim\irN(A-I)\notin\N$.
\end{proof}

Now we turn to the case when $\dim\irH>\aleph_0$. If $T$ is a contraction on $\irH$, then $\irH$ can be decomposed into the orthogonal sum of separable reducing subspaces $\irH = \sum_{\xi\in \Xi}\oplus\irH_\xi$ and so $T=\sum_{\xi\in \Xi}\oplus T_\xi$, where $T_\xi=T|\irH_\xi$. Hence $A_T$ is the orthogonal sum of asymptotic limits of contractions, all acting on a separable space: $A_T = \sum_{\xi\in \Xi}\oplus A_{T_\xi}$.

If $\kappa$ is an infinite cardinal number, satisfying $\kappa\leq\dim\irH$, then the closure of the set $\irE_\kappa := \{S\in\irB(\irH)\colon \dim (\irR(S))^-<\kappa\}$, is a proper two-sided ideal, denoted by $\irC_\kappa$. Let $\irF_\kappa := \irB(\irH)/\irC_\kappa$ be the quotient algebra, $\pi_\kappa\colon \irB(\irH)\to\irF_\kappa$ the quotient map and $\|.\|_\kappa$ the quotient norm on $\irF_\kappa$. For an operator $A\in\irB(\irH)$ we use the notation $\|A\|_\kappa := \|\pi_\kappa(A)\|_\kappa$, $\sigma_\kappa(A) := \sigma(\pi_\kappa(A))$ and $r_\kappa(A) := r(\pi_\kappa(A))$. (For $\kappa=\aleph_0$ we get the ideal of compact operators, $\|A\|_{\aleph_0} = \|A\|_e$ is the essential norm, $\sigma_{\aleph_0}(A)=\sigma_e(A)$ and $r_{\aleph_0}(A)=r_e(A)$.) For more details see \cite{terElst} or \cite{Luft}.

\begin{thm}\label{main_thm_arbitrary}
Let $\irH$ be a non-separable Hilbert space and let $A$ be a positive contraction acting on it. Then the following four conditions are equivalent:
\begin{itemize}
\item[\textup{(i)}] $A$ arises asymptotically from a contraction,
\item[\textup{(ii)}] $A$ arises asymptotically from a contraction in uniform convergence,
\item[\textup{(iii)}] $A$ is a finite rank projection, or $\kappa=\dim\irH((0,1])\geq\aleph_0$ and $r_\kappa(A)=1$ holds,
\item[\textup{(iv)}] $\dim{\irH\left((0,1]\right)}=\dim{\irH\left((\delta,1]\right)}$ for any $0\leq\delta<1$.
\end{itemize}
Moreover, when $\dim\irN(A-I)=0$ is zero or infinite and (iv) holds, then we can choose a $C_{\cdot 0}$ contraction $T$ such that $A$ is the uniform asymptotic limit of $T$.
\end{thm}

\begin{proof} We may suppose that $A$ is not a projection of finite rank. Since $T=\sum_{\xi\in \Xi}\oplus T_\xi$, where every $T_\xi$ acts on a separable space, one can obtain by Theorem \ref{mainthm_separable} that (i) implies (iv). The implication (ii)$\Longrightarrow$(i) is obvious. 

For the implication (iv)$\Longrightarrow$(ii) (which completes the chain (i)$\Longrightarrow$(iv)$\Longrightarrow$(ii)$\Longrightarrow$(i)), set $\alpha=\dim\irH((0,1])$, which is necessarily infinite. If $\alpha=\aleph_0$, then applying Theorem \ref{mainthm_separable} we can get $A$ as the uniform asymptotic limit of a contraction (on the nullspace of $A$ we take the zero operator). So we may suppose that $\alpha>\aleph_0$. We may assume also that $A$ is injective. Now we take an arbitrary strictly increasing sequence $0 = a_0 < a_1 < a_2 < \dots$ such that $\lim_{j\to\infty}a_j=1$, and let $\alpha_j = \dim\irH((a_j,a_{j+1}])$ for every $j\in\Z_+$. Obviously $\beta := \sum_{j=0}^\infty \alpha_j =\dim\irH((0,1))\leq\dim\irH((0,1])=\alpha$. Clearly one of the following four cases occurs.

\bigskip

\textit{Case 1: $\alpha_j=\alpha$ for infinitely many indices $j$.} Then without loss of generality, we may suppose that this holds for every index $j$. By Lemma \ref{block_diag_same_dim_lemma} we can choose a contraction $T_0\in \irB(\irH((0,1)))$ such that $\|T_0^{*n}T_0^{n}-A|\irH((0,1))\| \to 0$. Set the operator
\[T := T_0\oplus V\in \irB\Big(\irH\big((0,1)\big)\oplus\irH\big(\{1\}\big)\Big),\] 
where $V\in\irB(\irH\big(\{1\}\big))$ is an arbitrary isometry. Trivially $T$ is a contraction with the uniform asymptotic limit $A$.

\bigskip

\textit{Case 2: $\dim\irH(\{1\}) = \alpha$.} Let us decompose $\irH(\{1\})$ into the orthogonal sum $\irH(\{1\}) = \big(\sum_{k=1}^\infty\oplus \irX_k\big) \oplus \irX$, where $\dim\irX_k=\beta$ for every $k\in\N$ and $\dim \irX = \alpha$. Setting $\irX_0 := \irH([0,1))$, we may apply Lemma \ref{block_diag_same_dim_lemma} for the restriction of $A$ to $\sum_{k=0}^\infty\oplus \irX_k$. Taking any isometry on $\irX$, we obtain that (ii) holds.

\bigskip

\textit{Case 3: $\dim\irH(\{1\})<\alpha$ and $\alpha_j<\alpha$ for every $j$.} Then clearly $\dim\irH((\delta,1)) = \dim\irH((0,1)) = \alpha$ for any $\delta\in[0,1)$. Joining subintervals together, we may assume that $\aleph_0\leq \alpha_j<\alpha_{j+1}$ holds for every $j\in\Z_+$ and $\sup_{j\geq 0}\alpha_j = \alpha$. Let $\irX_j := \irH((a_{j},a_{j+1}])$ for every $j\in\Z_+$. Obviously we can decompose every subspace $\irX_j$ into an orthogonal sum $\irX_j = \sum_{k=0}^j \oplus \irX_{j,k}$ such that $\dim\irX_{j,k} = \alpha_{k}$ for every $0 \leq k \leq j$. Then by Lemma \ref{block_diag_same_dim_lemma} we obtain a contraction $T_k \in \irB(\sum_{j=k}^\infty \oplus \irX_{j,k})$ such that the asymptotic limit of $T_k$ is $A\big|\sum_{j=k}^\infty \oplus \irX_{j,k}$ in uniform convergence. In fact, from the proof of  Lemma \ref{block_diag_same_dim_lemma}, one can see that 
\[ \Bigg\|T_k^{*n}T_k^n - A\bigg|\sum_{j=k}^\infty\oplus \irX_{j,k}\Bigg\| \leq \frac{1}{a_{n+k}}-1 \leq \frac{1}{a_{n}}-1\to 0. \] 
Therefore, if we choose an isometry $V\in\irB(\irH(\{1\}))$, we get that (ii) is satisfied with the contraction $T := \big(\sum_{k=0}^\infty\oplus T_k\big)\oplus V \in\irB(\irH)$.

\bigskip

\textit{Case 4: $\dim\irH(\{1\})<\alpha$ and $\alpha_j = \alpha$ holds for finitely many $j$ (but at least for one).} We may assume $\alpha_0 = \alpha$, $\aleph_0\leq\alpha_j<\alpha_{j+1}$ for every $j\in\N$ and $\sup_{j\geq 1}\alpha_j = \alpha$. Take an orthogonal decomposition $\irH((0,a_1)) = \sum_{k=1}^\infty \oplus \irL_k$, where $\dim\irL_k = \alpha_k$. Set also $\irX_j := \irH((a_{j},a_{j+1}])$ for every $j\in\N$ and take a decomposition $\irX_j = \sum_{k=1}^j \oplus \irX_{j,k}$ such that $\dim\irX_{j,k} = \alpha_{k}$ for every $1 \leq k \leq j$. Thus by Lemma \ref{block_diag_same_dim_lemma} we obtain a contraction $T_k \in \irB(\irL_k\oplus\sum_{j=k}^\infty \oplus \irX_{j,k})$ such that the asymptotic limit of $T_k$ is the restriction of $A $ to the subspace $\irL_k\oplus\sum_{j=k}^\infty \oplus \irX_{j,k}$ in uniform convergence. As in Case 3, we get (ii).

\bigskip

Now we turn to the implication (iii)$\Longrightarrow$(iv). Since $1 = r_\kappa(A) \leq \|A\|_\kappa \leq \|A\|\leq 1$, we have $\|A\|_\kappa = 1$. Applying Lemma 5 in \cite{terElst} we get $\dim\irH((\delta,1]) = \kappa$ for all $0\leq\delta<1$. For the implication (iv)$\Longrightarrow$(iii) we may assume that $\dim\irH((0,1])\geq\aleph_0$. Again applying Lemma 5 in \cite{terElst}, we get $\|A^n\|_\kappa=1$ for all $n\in\N$. This means that $r_\kappa(A)=1$.

Finally we notice that if $\dim\irN(A-I)\notin\N$, then we can choose a $C_{\cdot 0}$-contraction.
\end{proof}

We conclude this Section with a Corollary. The proof is immediate from condition (iv) of the last theorem, so we omit it.

\begin{cor}
Suppose that the function $g\colon[0,1]\to[0,1]$ is continuous, increasing, $g(0)=0$, $g(1)=1$ and $0<g(t)<1$ for $0<t<1$. If $A$ arises asymptotically from a contraction, then  so does $g(A)$.
\end{cor}

For example, if $A$ is an asymptotic limit of a contraction, then $A^q$ is also an asymptotic limit of a contraction for every $0<q$.

\section{Contractions with coinciding asymptotic limits}

In this section we provide conditions for two contractions to have the same asymptotic limit. We show among others that a non-constant inner function of a completely non-unitary (c.n.u.) contraction $T$ has the same asymptotic limit as $T$. We recall that $T$ is a c.n.u. contraction, if only the zero subspace reduces $T$ to a unitary operator. In connection with the Sz.-Nagy--Foias functional calculus we refer to \cite{NFBK}. We relate also the asymptotic limit of the product of two contractions to the asymptotic limit of the contractions. First we consider the case when the contractions commute.

\begin{prop} \label{commuting_contr_aslim}
If $T_1, T_2\in\irB(\irH)$ are commuting contractions, then
\[ A_{T_1T_2}\leq A_{T_1} \text{ and } A_{T_1T_2}\leq A_{T_2}. \]
Consequently $\irH_0(T_1)\vee\irH_0(T_2)\subset\irH_0(T_1T_2)$.
\end{prop}

\begin{proof} For an arbitrary vector $x\in\irH$ and $i=1,2$, we have
\[ \big<A_{T_1T_2}x,x\big> = \lim_{n\to\infty}\big<(T_1T_2)^{*n}(T_1T_2)^{n}x,x\big> = \lim_{n\to\infty}\|(T_1T_2)^{n}x\|^2 \]
\[\leq \lim_{n\to\infty}\|T_i^{n}x\|^2 = \lim_{n\to\infty}\big<T_i^{*n}T_i^{n}x,x\big> = \big<A_{T_i}x,x\big>, \]
where we used the commuting property in the step $(T_1T_2)^n = T_1^nT_2^n = T_2^nT_1^n$.
\end{proof}

If $a\in\D$, then $b_a(z) := \frac{z-a}{1-\overline{a}z}$ is the so called Möbius transformation. It is a Riemann mapping from $\D$ onto itself. We use the notation $T_a:=b_a(T)$. It is easy to see that $b_{-a}(T_a) = b_{-a}(b_{a}(T)) = (b_{-a}\circ b_{a})(T) = T$.

\begin{lem} \label{cnu_Mobius}
If $T$ is a c.n.u. contraction, then $A_T=A_{T_a}$
\end{lem}

\begin{proof}
Consider the realization of the isometric asymptotes $(X^+_T,V_T)$ and $(X^+_{T_a},V_{T_a})$ of $T$ and $T_a$, respectively, where $X_T^+x = A_T^{1/2}x$ and $X_{T_a}^+x = A_{T_a}^{1/2}x$ ($x\in\irH$). Obviously $(X^+_T,b_a(V_T))$ is a contractive intertwining pair for $T_a$, hence we have a unique contractive transformation $Z$ such that $ZV_{T_a} = b_a(V_T) Z$ and $X^+_T = ZX^+_{T_a}$. Since
\[ \big<A_Tx,x\big> = \big<X^+_Tx,X^+_Tx\big> = \|X^+_Tx\|^2 = \|ZX^+_{T_a}x\|^2 \leq \|X^+_{T_a}x\|^2 = \big<A_{T_a}x,x\big>, \]
for every $x\in\irH$, it follows that $A_{T} \leq A_{T_a}$. Then
\[ A_T \leq A_{T_a} \leq A_{(T_a)_{-a}} = A_T, \]
which gives what we wanted.
\end{proof}

Now we concentrate on inner functions of $T$.

\begin{thm}
If $u$ is a non-constant inner function and $T$ is a c.n.u. contraction, then $A_T = A_{u(T)}$.
\end{thm}

\begin{proof}
Set $a := u(0)$ and $v = b_a\circ u$. Then obviously $v = \chi w$, where $\chi(z)=z$ and $w$ is an inner function. From Proposition \ref{commuting_contr_aslim} and Lemma \ref{cnu_Mobius} we get $A_{u(T)}=A_{v(T)}\leq A_T$. We consider isometric asymptotes $(X^+_T,V_T)$ and $(X^+_{u(T)},V_{u(T)})$ of $T$ and $u(T)$, respectively. The pair $(X^+_T,u(V_T))$ is a contractive intertwining pair of $u(T)$. Using the universal property of the isometric asymptote, we get a unique contractive transformation $Z$ such that $ZV_{u(T)} = u(V_T) Z$ and $X^+_T = ZX^+_{u(T)}$. The last equality implies $A_T \leq A_{u(T)}$, and so $A_T = A_{u(T)}$.
\end{proof}

For an alternative proof of the previous statement see Lemma III.1 in \cite{CassierFack}. It can be also derived from Theorem 2.3 in \cite{KerchyDouglasAlg}. The next theorem is a complement of Proposition \ref{commuting_contr_aslim} in a certain revise.

\begin{thm} \label{same_aslim}
Let $T_1, T_2$ be contractions such that $A_{T_1} = A_{T_2} = A$. Then $A \leq A_{T_1T_2}$.
\end{thm}

\begin{proof}
Set $X^+ \colon \irH\to\irR(A)^-$, where $X^+h := A^{1/2}h$, and consider the isometric asymptotes
\[(X^+, V_1), (X^+, V_2) \text{ and } (X^+_{T_1T_2},W)\]
of $T_1$, $T_2$ and $T_1T_2$, respectively. Obviously the pair $(X^+,V_1V_2)$ is a contractive intertwining pair for $T_1T_2$. Hence we get, from the universality property of the isometric asymptote, that there is a unique contractive $Z$ with the property $ZW = V_1V_2Z$ and $X^+ = ZX^+_{T_1T_2}$. Therefore $A \leq A_{T_1T_2}$.
\end{proof}

\begin{cor} \label{same_aslim_com}
When the contractions $T_1, T_2$ commute and $A_{T_1} = A_{T_2}$, then $A_{T_1T_2} = A_{T_1} = A_{T_2}$.
\end{cor}

\begin{proof}
This is an immediate consequence of Theorem \ref{same_aslim} and Proposition \ref{commuting_contr_aslim}.
\end{proof}

Concluding the paper we provide two examples. First we give two contractions $T_1,T_2\in C_{1\cdot}(\irH)$ such that $A_{T_1} = A_{T_2}$ and $A_{T_1T_2} \neq A_{T_1}$. This shows that Theorem \ref{same_aslim} cannot be strengthened to equality even in the $C_{1\cdot}$ case. By Corollary \ref{same_aslim_com} these contractions don't commute.

\begin{exmpl}
\textup{Take an orthonormal basis $\{e_{i,j}\colon i,j\in\N\}$ in $\irH$. The operators $T_1,T_2 \in\irB(\irH)$ are defined in the following way:}
\[ T_1 e_{i,j} := \left\{ \begin{matrix}
e_{i,j+1} & \text{if } j=1 \\
\frac{\sqrt{j^2-1}}{j} e_{i,j+1} & \text{if } j>1
\end{matrix}
\right., \quad T_2 e_{i,j} := \left\{ \begin{matrix}
e_{1,2} & \text{if } i=j=1 \\
e_{i+1,j-1} & \text{if } j=2 \\
\frac{\sqrt{3}}{2} e_{i-1,j+2} & \text{if } i>1, j=1 \\
\frac{\sqrt{j^2-1}}{j} e_{i,j+1} & \text{if } j>2
\end{matrix}
\right.. \]
\textup{$T_1$ and $T_2$ are orthogonal sums of infinitely many contractive, unilateral weighted shifts, with different shifting schemes. Straightforward calculations yield that}
\[ A_{T_1} e_{i,j} = A_{T_2} e_{i,j} = \left\{ \begin{matrix}
\frac{1}{2} e_{i,j} & \text{if } j = 1 \\
\frac{j-1}{j} e_{i,j} & \text{if } j > 1 
\end{matrix}
\right., \]
\textup{for every $i,j\in\N$, since $\Big(\prod_{l=j}^\infty \frac{\sqrt{l^2-1}}{l}\Big)^2 = \frac{j-1}{j}$ for $j>1$. On the other hand}
\[ T_2T_1 e_{i,j} = \left\{ \begin{matrix}
e_{i+1,1} & \text{if } j=1 \\
\frac{\sqrt{j^2-1}}{j}\frac{\sqrt{(j+1)^2-1}}{j+1}e_{i,j+2} & \text{if } j>1
\end{matrix}
\right., \]
\textup{hence}
\[ A_{T_2T_1} e_{i,j} = \left\{ \begin{matrix}
e_{i,1} & \text{if } j=1 \\
\frac{j-1}{j}e_{i,j} & \text{if } j>1
\end{matrix}
\right.,\]
\textup{since} 
\[ \bigg(\prod_{m=0}^\infty \frac{\sqrt{(j+2m)^2-1}}{j+2m}\frac{\sqrt{(j+1+2m)^2-1}}{j+1+2m}\bigg)^2 = \bigg(\prod_{l=j}^\infty \frac{\sqrt{l^2-1}}{l}\bigg)^2 = \frac{j-1}{j}.\] 
\textup{Therefore $A_{T_1}\leq A_{T_1T_2}$ and $A_{T_1}\neq A_{T_1T_2}$.}
\end{exmpl}

Finally, we give two contractions $T_1,T_2\in C_{0\cdot}(\irH)$ such that $T_1T_2\in\irC_{1\cdot}(\irH)$.

\begin{exmpl}
\textup{Take the same orthonormal basis in $\irH$ as in the previous example. The $C_{0\cdot}$-contractions $T_1,T_2\in\irB(\irH)$ are defined by}
\[ T_1 e_{i,j} := \frac{\sqrt{(i+1)^2-1}}{i+1} e_{i,j+1}, \qquad 
T_2 e_{i,j} := \left\{ \begin{matrix}
0 & \text{if } j=1 \\
e_{i-1,j+1} & \text{if } j>1
\end{matrix} \right.. \]
\textup{By a straightforward calculation we can check that}
\[ T_2T_1 e_{i,j} = \frac{\sqrt{(i+1)^2-1}}{i+1} e_{i+1,j}, \]
\textup{and so}
\[ A_{T_2T_1} e_{i,j} = \frac{i}{i+1} e_{i,j}. \]
\textup{Thus $T_1T_2\in C_{1\cdot}(\irH)$.}
\end{exmpl}

\noindent\textit{\textbf{Acknowledgement.}} The author is very grateful to professor {\sc L. Kérchy} for his useful suggestions. Also many thanks to {\sc A. F. M. ter Elst} for recommending the articles \cite{terElst} and \cite{Luft}.

Bolyai Institute, University of Szeged

Aradi vértanúk tere 1

H-6720, Szeged, HUNGARY

{\sc e-mail}: gehergy@math.u-szeged.hu

\end{document}